\documentclass[11pt]{amsart}

\usepackage{amsmath}
\usepackage{amssymb}
\usepackage{amsthm}
\usepackage[all]{xy}
\usepackage[OT2,T1]{fontenc}
\usepackage{ae}
\usepackage{aecompl}

\DeclareSymbolFont{cyrletters}{OT2}{wncyr}{m}{n}
\DeclareMathSymbol{\Sha}{\mathalpha}{cyrletters}{"58}

\evensidemargin \oddsidemargin
\newtheorem{theorem}{Theorem}[section]
\newtheorem{corollary}[theorem]{Corollary}
\newtheorem{lemma}[theorem]{Lemma}
\newtheorem{conjecture}[theorem]{Conjecture}
\newtheorem{proposition}[theorem]{Proposition}

\newenvironment{remarks}{\noindent {\bf Remarks.}}{}

\newcommand{\Q}{\mathbb Q}
\newcommand{\R}{\mathbb R}
\newcommand{\C}{\mathbb C}
\newcommand{\Z}{\mathbb Z}
\newcommand{\F}{\mathbb F}
\newcommand{\Ok}{{\mathcal O}_K}

\DeclareMathOperator{\Spec}{Spec}
\DeclareMathOperator{\Hom}{Hom}
\DeclareMathOperator{\Det}{Det}
\DeclareMathOperator{\Fil}{Fil}

\begin{document}
\title[Base change of elliptic curves]{Equivariant Birch-Swinnerton-Dyer conjecture for the base change of
elliptic curves: An example}
\author{Tejaswi Navilarekallu}
\address{Department of Mathematics\\
Indian Institute of Science\\
Bangalore 560012 India.}
\date{}
\begin{abstract}
Let $E$ be an elliptic curved defined over $\Q$ and let $K/\Q$ be a finite
Galois extension with Galois group $G$.  The equivariant
Birch-Swinnerton-Dyer conjecture for $h^1(E\times_{\Q} K)(1)$ viewed as a
motive over $\Q$ with coefficients in $\Q[G]$ relates the twisted $L$-values
associated with $E$ with the arithmetic invariants of the same.  In this
paper we prescribe an approach to verify this conjecture for a given data.
Using this approach, we verify the conjecture for an elliptic curve of
conductor 11 and an $S_3$-extension of $\Q$.
\end{abstract}
\maketitle{}

\section{Introduction}

Let $E$ be an elliptic curve defined over $\Q$.  Let $K/\Q$ be a finite
Galois extension with Galois group $G$ and let $E_K := E \times_{\Spec \Q}
\Spec K$.  Our interest is in the motive $M := h^1(E_K)(1)$ which has
a natural action of the semisimple algebra $\Q[G]$.  We regard $M$
as a motive defined over $\Q$ with coefficients in $\Q[G]$.  For a ring
$S$, we let $\zeta(S)$ denote it's centre.  Let
$$ \Xi = \Det_{\Q[G]} E(K)_{\Q} \otimes \Det^{-1}_{\Q[G]} E(K)_{\Q}^*
\otimes \Det^{-1}_{\Q[G]} H^1((E_K)(\C),\Q)^+ \otimes \Det_{\Q[G]}
H^0(E_K,\Omega_{E_K}^1)^*, $$
where $\Omega_{E_K}^1$ is the sheaf of differentials and $\Det_{\Q[G]}$
is a $\zeta(\Q[G])^{\times}$-module valued functor introduced below.  There
is an isomorphism
$$\vartheta_{\infty} : \zeta(\R[G]) \simeq \Xi \otimes \R,$$
given by the height pairing and the period isomorphism attached to the
elliptic curve.  The rationality part of the equivariant conjecture says
that the special value $L^*(M,0)^{-1}$, which can be viewed as an element of
$\zeta(\R[G])$, maps under $\vartheta_{\infty}$ into an element of $\Xi
\otimes 1$.

Let $l$ be a rational prime and let $S_l$ be the finite set of primes in
$\Q$ consisting of primes of bad reduction, ramified primes, infinite prime
and $l$.  There exists a perfect complex $R\Gamma_c(\Z_{S_l},
H_l(M))$ of $\Q_l[G]$-modules along with an isomorphism
$$ \vartheta_l: \Xi(M) \otimes \Q_l \simeq \Det_{\Q_l[G]} R\Gamma_c(
\Z_{S_l},H_l(M))$$
of $\Q_l[G]$-modules.

Let $T_l := \mathrm{Ind}_K^{\Q} \left( \varprojlim_n
E(\overline{\Q})[l^n] \right)$, a $\Z_l$-lattice in $V_l := H_l(M)$.
Then, $R\Gamma_c(\Z_{S_l},T_l)$ is a perfect complex of $\Z_l[G]$-modules
with $R\Gamma_c(\Z_{S_l},T_l) \otimes_{\Z_l[G]} \Q_l[G] \simeq
R\Gamma_c(\Z_{S_l},V_l)$.  This along with a trivialization $\tau_l : \Xi
\otimes \Q_l \simeq \zeta(\Q_l[G])$ gives an element $\xi_l =
([R\Gamma_c(\Z_{S_l},T_l)];\tau_l) \in K_0(\Z_l[G];\Q_l)$.  On the other
hand, the leading coefficient $L^*(M,0)$ of the motivic $L$-function at $s=0$
gives a class $\theta_l \in K_0(\Z_l[G];\Q_l)$, obtained via the long exact
sequence of $K$-theory.  The equivariant conjecture formulated by Burns and
Flach states that $\xi_l - \theta_l$ vanishes in $K_0(\Z_l[G];\Q_l)$ for all
primes $l$.

In this paper, we present a technique to verify the above conjecture for a
given elliptic curve and a fixed extension under certain hypotheses.  While
the equality in the conjecture is exactly checked in the commutative case,
we only give a numerical verification in the noncommutative case.  We prove

\begin{theorem} \label{intro_theorem}
Let $E$ be the elliptic curve $y^2 + xy + y = x^3 + 11$ and let $K$ be the
splitting field of $x^3 -4x + 3$ over $\Q$.  Let $G (\simeq S_3)$ denote the
Galois group of $K$ over $\Q$.  Let $l$ be an odd rational prime.
If $|\Sha(E/K)|_l = 1$ then the $l$-part of the equivariant conjecture holds
numerically for the motive $h^1(E_K)(1)$ and the $\Z$-order $\Z[G]$.
\end{theorem}

We analyse the complex $R\Gamma_c$ via the exact triangle
\begin{eqnarray} \label{exact_t}
R\Gamma_c(\Z_{S_l},T_l) \rightarrow R\Gamma_f(\Q,T_l) \rightarrow
\oplus_{p \in S_l} R \Gamma_f(\Q_p,T_l),
\end{eqnarray}
where the cohomology of the complexes $R\Gamma_f(\Q,T_l)$ can be described
in terms of the group of points and the Tate-Shafarevich group of $E$.
However, to ensure that all the terms in the triangle (\ref{exact_t}) are
perfect complexes of $\Z_l[G]$-modules we have to make certain noncanonical
choices (see Section \ref{arithmetic} below).

On the analytic side, we use the theory of modular symbols to write down the
special value $L(E\otimes \chi,1)$ for a Dirichlet character $\chi$.  However,
no such theory exists for elliptic curves defined over a general number field,
or for twists by nonabelian characters.  We therefore numerically compute the
special values using methods of Tollis as explained in \cite{td}.

The paper is organized as follows.  In Section 2, we give a brief description
of the conjecture specialized to the case we are interested in.  We then
indicate methods to compute the arithmetic and the analytic sides of the
conjecture in Section 3 and 4 respectively.  We prove the main result in the
last section.

\section{Equivariant conjecture}

\subsection{Algebraic $K$-groups}
Let $R$ and $S$ be rings, and let $\phi: R \rightarrow S$ be a ring
homomorphism.  Let $K_i(*)$ denote the associated $K$-groups.  We have
the following long exact sequence
$$ K_1(R) \longrightarrow K_1(S) \stackrel{\delta}{\longrightarrow}
K_0(R;S) \longrightarrow K_0(R) \longrightarrow K_0(S)$$
where $K_0(R;S)$ is the relative $K_0$-group.

In terms of generators and relations, $K_0(R;S)$ is generated by triples
$(M,N;\lambda)$ where $M, N$ are finitely generated projective $R$-modules
and $\lambda: M \otimes S \rightarrow N \otimes S$ is an $S$-isomorphism,
and with relations given by short exact sequences (see \cite{cr}).

Similarly, for any ring $R$, the group $K_1(R)$ is generated by pairs
$(M;\phi)$ where $M$ is a free $R$-module of finite rank and $\phi$ is
an $R$-autorphism of $M$, and with relations given by the inclusions
of free $R$-modules.
We use these descriptions to denote elements of $K_1$ and the relative
$K_0$-group.

For a $\Q$-algebra $A$ and a $\Z$-order $\mathfrak{A}$ in $A$, we denote
by $Cl(\mathfrak{A})$ the associated class group.  That is,
$$ Cl(\mathfrak{A}) := \mathrm{ker}\left( K_0(\mathfrak{A}) \rightarrow
K_0(A) \right).$$
We use analogous notation in the local case as well.

\subsection{Virtual objects} \label{virtual_objects}

Let $R$ be a ring and let $\mathrm{PMod}(R)$ denote the category of finitely
generated projective $R$-modules.  In \cite{deligne}, Deligne has constructed
a Picard category $V(R)$ of virtual objects and a universal determinant
functor
$$[~] : \mathrm{PMod}(R) \rightarrow V(R) $$
satisfying certain conditions.  This functor naturally extends to a functor
$$ [~]: D^p(R) \rightarrow V(R) $$
where $D^p(R)$ is the category of perfect complexes of $R$-modules (see
\cite{bf1,navil2} for details).

It follows from the proof of the existence of $V(R)$ in \cite{deligne} that
there are isomorphisms
\begin{eqnarray}
K_i(R) \stackrel{\thicksim}{\longrightarrow} \pi_i (V(R))
\end{eqnarray}
for $i = 0,1$ where $K_i(R)$ denotes the algebraic $K$-group associated to
$R$, and $\pi_0(V(R))$ is the group of
isomorphism classes of objects of $V(R)$ and $\pi_1(V(R)) =
\mathrm{Aut}_{V(R)} (1_{V(R)})$.

Given a finitely generated subring $S$ of $\Q$, an $S$-order $\mathfrak{A}$
of a finite dimensional $\Q$-algebra $A$ is a finitely generated $S$-module
such that $\mathfrak{A} \otimes_{\Z} \Q = A$.  For any such $S$-order, and
any field extension $F$ of $\Q$, one has a notion of relative virtual objects
$V(\mathfrak{A};F)$ and $V(\mathfrak{A}_p; \Q_p)$, where $\mathfrak{A}_p =
\mathfrak{A} \otimes_{\Z} \Z_p$ (see \cite{bf1} for details).  There are
isomorphisms
\begin{eqnarray}
\pi_0( V(\mathfrak{A};F)) \stackrel{\thicksim}{\longrightarrow}
K_0(\mathfrak{A};F)
\end{eqnarray}
and
\begin{eqnarray} \label{rel_pi0}
 \pi_0( V(\mathfrak{A}_p;\Q_p)) \stackrel{\thicksim}{\longrightarrow}
K_0(\mathfrak{A}_p;\Q_p),
\end{eqnarray}
which are compatible with the Mayor-Vietoris sequences (cf. Prop. 2.5 in
\cite{bf1}).  Given a perfect complex $C_{\bullet}$ of
$\mathfrak{A}$-modules
and an isomorphims $\tau: [C_{\bullet} \otimes_{\mathfrak{A}} F]
\simeq 1_{V(F)}$, we denote by $([C_{\bullet}];\tau)$ the corresponding
element of $K_0(\mathfrak{A};F)$ under the above isomorphism.  We use
analogous notation in the local case as well.

\noindent {\bf Remarks.}
\begin{enumerate}
\item Let $R$ be a commutative ring.  Then an $R$-module $P$ is
projective if and only if it is locally free at all the prime ideals of $R$.
In this case, the determinant functor can be defined locally by
$$ \Det_{R_{\mathfrak{p}}}(P_{\mathfrak{p}}) =
\left( \wedge^{\mathrm{rank}_{R_{\mathfrak{p}}}(P_{\mathfrak{p}})}
P_{\mathfrak{p}}, \mathrm{rank}_{R_{\mathfrak{p}}}(P_{\mathfrak{p}}) \right) $$
for every prime $\mathfrak{p} \in \Spec(R)$.  Note that thus defined
$\Det_R(P)$ is a graded line bundle.  Let $\mathcal{P}(R)$ denote the category
of graded line bundles over $R$.  Then the universal property of $V(R)$
defines a tensor functor
$f_R: V(R) \rightarrow \mathcal{P}(R)$.  The functor $f_R$ is an equivalence
if $R$ is either semisimple or a finite flat $\Z$-algebra.  For such rings,
we assume that the determinant functor is constructed as above.

\item Even if $R$ is not commutative, but is semisimple, one can construct the
determinant functor in a similar fashion by looking at the indecomposable
idempotents.  We give this construction below since it is key to some of our
computations.

By Wedderburn's decomposition we can assume that $R$ is a central simple
algebra over a field $F$.  So $R \simeq M_n(D)$ for some division ring $D$
with centre $F$.  Further, by fixing an exact Morita equivalence
$\mathrm{PMod}(M_n(D)) \rightarrow \mathrm{PMod} (D)$, we may assume that
$R = D$.  Fix a field extension $F'/F$ such that $D \otimes_F F' \simeq
M_d(F')$.  Let $e$ be an indecomposable idempotent of $M_d(F')$ and let
$e_1, \ldots, e_d$ be an ordered $F'$-basis of $eM_d(F')$.  Let $V$ be a
finitely generated projective (and hence free) $D$-module.  Let $\{v_1,
\ldots, v_r \}$ be a $D$-basis of $V$.  Set $b := \wedge e_iv_j$.  This is
an $F'$-basis of $\Det_{F'} (e(V \otimes_F
F'))$.  Since any change in the basis $\{ v_i \}_i$ multiplies $b$ by
an element of $F^{\times}$, the $F$-space spanned by $b$ yields a
well-defined graded $F$-line bundle.  This defines the determinant functor.
Thus, for a semisimple ring, the determinant functor can be constructed
as taking values in graded line bundles over the centre.  We use this
concrete construction when referring to the functor over $\R[G], \Q[G]$
and $\Q_l[G]$.

\end{enumerate}

\subsection{The motive $h^1(E_K)(1)$}
Let $K/\Q$ be a finite Galois extension with Galois group
$G$.  Let $E$ be an elliptic curve defined over $\Q$.  Denote by $E_K$ the
base change $E \times_{\Spec \Q} \Spec K$.  The Galois group $G$ acts on
the motive $M = h^1(E_K)(1)$ and thus the $\Q$-algebra $\Q[G]$ acts on
the realizations of $M$.  We regard $M$ as a motive defined over $\Q$
with coefficients in $\Q[G]$.  Note that the algebra $\Q[G]$ acts on
$$ H_{dR} := H_{dR}^1(E_K)/Fil^1 \simeq H^0(E_K,\Omega^1_{E_K})^*, $$
and on
$$ H_B^+ := \oplus_{v \in S_{\infty}(K)} H^1(vE(\C),2\pi i \Q)^{G_v},$$
where $*$ denotes the dual, and $S_{\infty}(K)$ denotes the set of infinite
places of $K$.  Here $H^1(vE(\C),*) = H^1(\sigma E(\C),*)$ for $\sigma \in
G$ corresponding to $v$.  Denoting by $c \in G$ the complex conjugation,
we identify $H^1(\sigma E(\C),*)$ and $H^1(c \circ \sigma E(\C),*)$ via
the isomorphism $\sigma E(\C) \simeq c \circ \sigma E(\C)$.

After identifying $H^1(vE(\C),*)$ with the dual of the
homology group, we define period isomorphism
\begin{eqnarray} \label{period_map}
\pi: H_B^+ \otimes_{\Q} \R \rightarrow H_{dR} \otimes_{\Q} \R
\end{eqnarray}
by
$$ \gamma \mapsto (\omega \mapsto \int_{\gamma} \omega) $$
for $\gamma \in \oplus_{v \in S_{\infty}(K)} H_1(vE(\C),\Q)^{G_v}$
and $\omega \in H^0(E_K,\Omega_{E_K}^1)$.

Let
$$ \Xi := \Det E(K)_{\Q} \otimes \Det^{-1} E(K)^*_{\Q} \otimes
\Det^{-1} H^1(E_K(\C),\Q)^+ \otimes \Det H^0(E_K,\Omega^1_{E_K})^*,$$
where all the determinants are taken over the ring $\Q[G]$.
Thanks to the above defined period isomorphism and the canonical height
pairing associated with the rational points on the elliptic curve, we get an
isomorphism
$$\vartheta_{\infty}: \Xi \otimes_{\Q} \R \simeq \zeta(\R[G])$$
of $\R[G]$-modules.

Let $l$ be a rational prime and let $S$ be the finite set of
primes in $\Q$ consisting of primes of bad reduction of $E$,
ramified primes in $K/\Q$ and the infinite prime.   Let $S_l := S
\cup \{ l \}$.  Also, let $S_{\infty} = \{ \infty \}$, $S_f = S \setminus
\{\infty\}$ and $S_{f,l} = S_l \setminus \{\infty\}$.  For any set $S_0$
of primes in $\Q$, we let $S_0(K)$ denote the primes in $K$ lying above
those in $S_0$.

Let $T_l := \mathrm{Ind}_K^{\Q} \left( \varprojlim_n
E(\bar{\Q})[l^n] \right)$ be the Tate module assocaited with $E_K$.
Then we have $V_l := H_l(M) \simeq T_l \otimes_{\Z_l} \Q_l$.  There exists
a perfect complex $R\Gamma_c(\Z_{S_l}, V_l)$ of $\Q_l[G]$-modules
along with an isomorphism
$$ \vartheta_l : \Xi \otimes_{\Q} \Q_l \simeq \Det_{\Q_l[G]}
R\Gamma_c(\Z_{S_l},V_l) $$
of $\Q_l[G]$-modules obtained via comparison isomorphisms
(see \cite{bf1,navil2} for details).

Analogously, one can define a perfect complex
$R\Gamma_c(\Z_{S_l}, T_l)$ of $\Z_l[G]$-modules with a mapping to
$R\Gamma_c(\Z_{S_l}, V_l)$ such that
$$ R\Gamma_c(\Z_{S_l}, T_l) \otimes_{\Z_l[G]} \Q_l[G] \simeq
R\Gamma_c(\Z_{S_l}, V_l). $$

\subsection{Special values of $L$-functions} \label{conj_formulation}
In this section we state the equivariant conjecture as formulated by
Burns and Flach.  See \cite{bf1} or \cite{navil2} for more details.

Let $\widehat{G}$ denote the set of irreducible complex characters of $G$.
The $\Q[G]$-equivariant $L$-function associated with the motive $h^1(E_K)(1)$
is the tuple
$$ L(h^1(E_K)(1),s) := \left( L(E\otimes \eta,s+1) \right)_{\eta \in
\widehat{G}}, $$
where $L(E\otimes \eta,s)$ is the twist of the Hasse-Weil $L$-function
attached to $E$.
For $\eta \in \widehat{G}$ let $L^*(E\otimes \eta,1)$ denote the leading
nonzero coefficient in the Talylor expansion of $L(E\otimes \eta,s)$ at
$s = 1$.  Then we let
$$ L^*(h^1(E_K)(1)) := \left(L^*(E\otimes \eta,1) \right)_{\eta \in
\widehat{G}}.$$
Note that $L(h^1(E_K)(1),s) \in \prod_{\eta \in \widehat{G}} \C \simeq \zeta(
\C[G])$ where $\zeta$ denotes the centre.  Hence, we can identify
$L^*(h^1(E_K)(1))$ with an element of $\zeta(\R[G])^{\times}$.  The rationality
part of the equivariant conjecture states that

\begin{conjecture} \label{rationality_conjecture}
With the notations as above, one has
$$ \vartheta_{\infty}^{-1}(L^*(h^1(E_K)(1))^{-1}) \in \Xi \otimes 1 \subset
\Xi \otimes \R.$$
\end{conjecture}

Now consider the reduced norm map
$$ \mathrm{nr}: K_1(\R[G]) \rightarrow \zeta(\R[G])^{\times}.$$
Note that the special value $L^*(h^1(E_K)(1)) \in \zeta(\R[G])^{\times}$.
By the weak approximation theorem, there exists $\lambda \in
\zeta(\Q[G])^{\times}$
such that $\lambda L^*(h^1(E_K)(1)) \in \mathrm{nr} (K_1(\R[G])$.  By the above
conjecture, we get an isomorphism
$$ \tau: \Xi \simeq \zeta(\Q[G])$$
which when tensored with $\R$ gives $-\mathrm{nr}^{-1}(\lambda
L^*(h^1(E_K)(1)) \cdot \vartheta_{\infty}$.  We therefore get
$$ R\Gamma_c(\Z_{S_l},V_l) \simeq \Xi \otimes \Q_l \simeq
\zeta(\Q_l[G]).$$
Denoting the composition by $\tau_l$ we have an element
$$ \xi_l := ([R\Gamma_c(\Z_{S_l},T_l)];\tau_l) \in K_0(\Z_l[G];\Q_l).$$

On the other hand, the local reduced norm map
$$ \mathrm{nr}_l: K_1(\Q_l[G]) \rightarrow \zeta(\Q_l[G])^{\times} $$
is an isomorphism.  Thus, via the connecting morphism $\delta_l : K_1(\Q_l[G])
\rightarrow K_0(\Z_l[G];Q_l)$ we get an element
$$ \theta_l := \delta_l(\mathrm{nr}_l^{-1}(\lambda)) \in K_0(\Z_l[G];\Q_l).$$

The integrality part of the equivariant conjecture says that

\begin{conjecture} \label{integrality_conjecture}
Assuming Conjecture \ref{rationality_conjecture}, one has that $ \xi_l - 
\theta_l$ vanishes in $K_0(\Z_l[G];\Q_l)$.
\end{conjecture}

Our aim is to prescribe a method to verify the above conjecture for a given
elliptic curve and a finite Galois extension.

\section{Arithmetic calculations} \label{arithmetic}

Henceforth we work under the following hypotheses.
\begin{itemize}
\item $K/\Q$ is a finite Galois extension with Galois group $G$.
\item $\Ok$ is a free $\Z[G]$-module of rank 1.  By a theorem of Taylor,
this condition holds if and only if $K/\Q$ is tamely ramified and $G$
has no symplectic character (cf. \cite{taylor}).
\item $E(K)$ and $\Sha(E/K)$ are finite.
\item $(\mathrm{cond}(E),\mathrm{disc}(K)) = 1$.  This ensures that the
deRahm cohomology with integral coefficients is a free $\Z[G]$-module.
\end{itemize}

\subsection{Period isomorphism}

Since $E(K)$ is torsion, we have
\begin{eqnarray} \label{ell_xi}
\Xi = \Det^{-1} H^1(E_K(\C),\Q)^+ \otimes \Det H^0(E_K,\Omega^1_{E_K})^*.
\end{eqnarray}
Let $\mathcal{E}$ be a N\'{e}ron model for $E$ over $\Z$.  Let $\omega_0$ be
a generator of $ H^0(\mathcal{E}, \Omega_{\mathcal{E}}^1)$.  Then the image
of the map
\begin{eqnarray*}
H_1(E(\C),\Z) & \rightarrow & \C \\
\gamma & \mapsto & \int_{\gamma} \omega_0
\end{eqnarray*}
is a $\Z$-lattice in $\C$. This lattice is generated by $\Omega^+$ and
$\Omega^-$, the real and purely imaginary periods associated with $E$
respectively (cf. \cite{silverman}).

Let $\alpha_0$ be a $\Z[G]$-generator for $\Ok$.  Fix an element $c$
of $G$ corresponding to complex conjugation (so $c$ is trivial if $K/\Q$
is real).  For a complex representation $\rho : G \rightarrow GL(V)$, we let
$d(\rho)_+$ and $d(\rho)_-$ denote the dimension of the eigenspaces of $V$,
with eigenvalues 1 and $-1$ respectively, for the action of $c$.

\begin{proposition} \label{period_iso}
In the above setting, the image of $\Xi$ in $\Xi \otimes \R \simeq
\zeta(\R[G])$ under the isomorphism $\vartheta_{\infty}$ is given by
$$ \Omega_+^{-d(\rho)_+} \cdot \Omega_-^{-d(\rho)_-} \cdot \left( \det \left(
\sum_{g \in G} g(\alpha_0) \rho_{\eta}(g^{-1}) \right) \right)_{\eta \in
\widehat{G}} \cdot \zeta(\Q[G]),$$
where $\rho_{\eta}$ is the representation corresponding to the character
$\eta$.
\end{proposition}

\begin{proof}
Recall that the isomorphism $\vartheta_{\infty}: \Xi \otimes_{\Q} \R \simeq
\zeta(\R[G])$ is constructed using the period map $\pi$ as in
(\ref{period_map}).  Therefore,
we shall first write down the Betti and deRham realizations, and the
corresponding period map between them.

Recall that $H_B^+ = \oplus_{v \in S_{\infty}} H^1(vE_K(\C),2\pi i
\Q)^{G_v}$.  We
shall identify each summand on the right hand side with the dual homology via
the isomorphism
$$ H^1(vE_K(\C),2 \pi i \Q) \simeq \Hom(H_1(vE_K(\C),\Q),2\pi i \Q).$$
Therefore, we have
$$ H_B^+ \simeq \Hom \left( \oplus_{v \in S_{\infty}} H_1(vE_K(\C),\Q),
(2 \pi i ) \Q \right)^{G_v}.$$
For a path $\gamma \in H_1(vE_K(\C),\Q)$, we let $\gamma^*$ denote the
corresponding element in the dual that maps $\gamma$ to 1 and the
orthogonal complement of $\gamma$ to 0.

Let $\gamma_+$ and $\gamma_-$ be $\Z$-generators of $H_1(E_K(\C),\Z)$ that
are eigenvectors, with eigenvalues 1 and $-1$ respectively, for the action
of the complex conjugation.  We also assume that these generators satisfy
$\int_{\gamma_+} \omega_0 = \Omega_+$ and $\int_{\gamma_-} \omega_0 =
\Omega_-$.  Note that, $\sigma \gamma_+$ and $\sigma \gamma_-$
are eigenvectors generating $H_1(vE_K(\C),\Z)$, for $v = v(\sigma)$.
Now define
$$\tilde{\gamma} : \oplus_{v \in S_{\infty}} H_1(vE_K(\C),\Q)  \rightarrow
(2 \pi i) \Q $$
by setting
\begin{eqnarray*}
\tilde{\gamma}(\sigma \gamma_*) & = & \left\{
\begin{array}{ll}
2 \pi i & \textrm{if $\sigma$=id, $*=\pm$} \\
0 & \textrm{otherwise.}
\end{array}
\right.
\end{eqnarray*}
It is easy to see that the restriction of $\tilde{\gamma}$ to $H_B^+$, which
we shall denote by the same symbol, generates $H_B^+$ as a $\Q[G]$-module.

Now, we consider the deRham realization.  Recall that by Serre duality one has
$$ H_{dR}/F^0 = H^1(E_K,\mathcal{O}_{E_K}) \simeq H^0(E_K,\Omega_{E_K}^1)^*$$
where $\Omega_{E_K}^1$ is the sheaf of differentials, and $*$ denotes the
dual.

\begin{lemma}
Let $\mathcal{E}$ and $\mathcal{E}_{\Ok}$ be the N\'{e}ron models for $E$
over $\Z$ and for $E_K$ over $\Ok$, respectively.  Suppose that the conductor
of $E$ and the discriminant of $K$ are relatively prime to each other.  Then
$$\mathcal{E}_{\Ok} \simeq \mathcal{E} \times_{\Spec(\Z)} \Spec(\Ok).$$
\end{lemma}

\begin{proof}
Note that the primes of bad reduction of $E$ and the primes that ramify in
$K/\Q$ do not intersect.  Therefore, $\mathcal{E} \times_{\Spec(\Z)} \Spec(\Ok)$
is an abelian scheme over $\Spec(\Ok)$.  The Lemma now follows from Corollary
1.4 in \cite{artin}.

\end{proof}

From the above Lemma we have
$$ H^0(E_K,\Omega_{E_K}^1) \simeq
H^0(\mathcal{E}_{\Ok},\Omega_{\mathcal{E}_{\Ok}}^1) \otimes_{\Ok} K
\simeq \left( H^0(\mathcal{E},\Omega_{\mathcal{E}}^1) \otimes_{\Z} \Ok
\right) \otimes_{\Ok} K.$$
Recall that $\omega_0$ is a generator for $H^0(\mathcal{E},
\Omega_{\mathcal{E}}^1)$.  Therefore,
$((\omega_0 \otimes \alpha_0) \otimes 1)^*$ is a $\Q[G]$-generator for
$H_{dR}/\Fil^0$.
 
After identifying the Betti and deRham realizations as above, the period map
$$ \pi : H_{B,\R}^+ \rightarrow H_{dR,\R}/\Fil^0$$
is given by
$$ \gamma^* \mapsto \left( \omega \otimes \alpha \mapsto \left( \int_{\gamma}
\omega \right)^{-1} \alpha \right), $$
for $\gamma \in \oplus_{v \in S_{\infty}} H_1(vE(\C),\Q)^{G_v}$, $\omega \in
H^0(\mathcal{E},\Omega_{\mathcal{E}}^1)$ and $\alpha \in K$.

Since $\Q[G]$ is semisimple, the determinant functor is given by the map
$\Det$ defined in Section \ref{virtual_objects}.  Let $\rho_{\eta}:G
\rightarrow M_d(F)$ be an irreducible complex representation of $G$ of
dimension $d$, with character $\eta$.  Here $F$ is a finite extension of
$\Q$.  By choosing a suitable basis, we assume that $\rho(c)$ is diagonal
with the first $d(\rho)_+$ diagonal entries being 1 and the rest of the
diagonal entries being $-1$.  We shall now compute the $\eta$-component of
the generator $(\tilde{\gamma}) \otimes ((\omega_0 \otimes \alpha_0)
\otimes 1)^*$ of $\Xi \otimes \R$.

Let $e_{11}$ be the indecomposable idempotent of $\Q[G]$ such that
$\rho(e_{11}) = (b_{ij})$ with $b_{ij} = 1$ for $(i,j) = (1,1)$ and $b_{ij}
= 0$ otherwise.  Let $e_{11}, e_{12}, ..., e_{1d}$ be an $F$-basis for
$e_{11}\Q[G]$.  We choose this basis such that
for $1 \leq r \leq d$, one has that $\rho_{\eta}(e_{1r})$ is the matrix
$(b_{ij})$ with $b_{ij} =1$ of $(i,j) = (1,r)$ and $b_{ij} = 0$ otherwise.

For $1 \leq k \leq d$ and $g \in G$, we have
\begin{eqnarray*}
\left( \pi(e_{1k}\tilde{\gamma}) \right) (\omega_0 \otimes g(\alpha_0))
& = & \Omega_{s(k)}^{-1} e_{1k}g(\alpha_0),
\end{eqnarray*}
where $s(k) = +$ if $1 \leq k \leq d(\rho)^+$ and $s(k) = -$ otherwise.
Therefore it follows that under the period isomorphism $\pi$, the image
of $e_{1k}\tilde{\gamma}$ is
$$\left( \Omega_{s(k)}^{-1} e_{1k} \sum_{g \in G} g(\alpha_0) g^{-1} \right)
(\omega_0 \otimes \alpha_0)^*.$$

Now, for an element $\theta \in \Q[G]$, one has that
$$ e_{1k} \theta = \sum_{r=1}^d a_{kr} e_{1r} $$
where $\rho_{\eta}(\theta) = (a_{ij})$.
Thus it follows that the image of $\wedge_{k=1}^d e_{1k}\tilde{\gamma}$
under the period isomorphism is $$ \Omega_+^{-d(\rho)_+} \cdot
\Omega_-^{-d(\rho)_-} \cdot \det \left( \sum_{g \in G}
g(\alpha_0) \rho_{\eta}(g^{-1}) \right) \cdot \wedge_{k=1}^d
e_{1k} (\omega_0 \otimes \alpha_0)^*.$$
The proposition now follows.
\end{proof}

\subsection{The cohomology of finite complexes} \label{finite_complexes}
We shall now define the complexes $R\Gamma_f$ and the exact triangle
(\ref{exact_t}). Let $T_l(E) := \varprojlim_n E(\bar{\Q})[l^n]$ and
$V_l(E) := T_l(E) \otimes_{\Z_l} \Q_l$.  Recall that $T_l :=
\mathrm{Ind}_K^{\Q} \left( T_l(E) \right)$ and $V_l :=
T_l \otimes_{\Z_l} \Q_l$ is the $l$-adic realization of $M = h^1(E_K)(1)$.
Note that $T_l$ is a Galois stable $\Z_l$-lattice in $V_l$.

Let $M_K := h^1(E_K)(1)$, regarded as a motive defined over $K$.  Then
$H_l(M_K) = V_l(E)$ and by Proposition 4.1 of \cite{bf} we have
$$ R\Gamma_c(\Z_{S_l},V_l) \simeq R\Gamma_c(\mathcal{O}_{K,S_l(K)},V_l(E)).$$
We will therefore consider the following exact triangle instead of
(\ref{exact_t}):
\begin{eqnarray} \label{exact_tk}
R\Gamma_c(\mathcal{O}_{K,S_l(K)},T_l(E)) \rightarrow R\Gamma_f(K,T_l(E))
\rightarrow \oplus_{v \in S_l(K)} R\Gamma_f(K_v,T_l(E)).
\end{eqnarray}
We can then set $R\Gamma(\Q,T_l) := R\Gamma(K,T_l(E))$ and
$R\Gamma_f(\Q_p,T_l) := \oplus_{v \mid p} R\Gamma_f(K_v,T_l(E))$ for
any $p \in S_l$.

Our aim is to define complexes $R\Gamma_f(K_v,T_l(E))$ for $v \in S_l(K)$
such that they satisfy the following:
\begin{itemize}
\item there is a map $R\Gamma_f(K_v,T_l(E)) \rightarrow R\Gamma(K_v,T_l(E))$
of complexes which upon tensoring
with $\Q_l$ becomes quasi-isomorphic to $R\Gamma_f(K_v,V_l(E)) \subseteq
R\Gamma(K_v,V_l(E))$;
\item $R\Gamma_f(K_p,T_l(E)) := \oplus_{v \mid p} R\Gamma_f(K_v,T_l(E))$ is  
a perfect complex of $\Z_l[G]$-modules.
\end{itemize}
We shall denote by
$R\Gamma_f(K_v,T_l)_{BK}$ the finite integral complexes defined by
Bloch-Kato (cf. \cite{bk}).  Note that these finite complexes in general
satisfy only the first condition and therefore doesn't fit our framework.  

\subsubsection{Case $v \in S_{\infty}(K)$}  In this case, one has that
$$ R\Gamma_f(K_v,V_l(E)) = R\Gamma(K_v,V_l(E)),$$
which is the standard complex of continuous cochains.  We therefore define
$R\Gamma_f(K_v,T_l(E))$ to be the complex $H^0(K_v,T_l(E))$ (concentrated
in degree zero).  Note that $R\Gamma_f(K_{\infty},T_l(E))$ is free of rank 1
over $\Z_l[G]$.  Further, under Grothendieck's comparison isomorphism,
$\oplus_{v \mid \infty} H^0(K_v,T_l(E))$ is mapped onto $\oplus_{v \in
S_{\infty}(K)} H^1(vE(\C),\Z)\otimes_{\Z} \Z_l$.

\subsubsection{Case $v \in S_f(K)$}
Note that for $v \in S_f(K)$, the complex $R\Gamma_f(K_v,V_l(E))$ is
\begin{eqnarray} \label{rgammaf}
V_l(E)^{I_v} \stackrel{1 - \mathrm{Fr}_v}{\longrightarrow} V_l(E)^{I_v},
\end{eqnarray}
where $I_v$ is the inertia group, $\mathrm{Fr}_v$ is the Frobenuis at $v$
and the modules are placed in degrees 0 and 1.  If $l$ does not divide the
order of the image of $I_v$ in $G$ then
$$ T_l(E)^{I_v} \stackrel{1 - \mathrm{Fr}_v}{\longrightarrow} T_l(E)^{I_v} $$
is a good choice for $R\Gamma_f(K_v,T_l(E))$.  The complex
(\ref{rgammaf}) has trivial cohomology and therefore we can
$R\Gamma_f(K_v,T_l(E))$ to be the zero complex whenever $l$ divides the
order of the image of $I_v$ in $G$.

\subsubsection{Case $v \mid l$}  In this case, one has
$$H^1_f(K_v,V_l(E)) \simeq (\varprojlim_n E(K_v)/l^n) \otimes_{\Z_l} \Q_l.$$
We first note that to define $R\Gamma_f(K_v,T_l(E))$ it is enough to define
a Galois stable $H_f^1(K_v,T_l(E)) \subset H_f^1(K_v,V_l(E))$ (cf.
\cite{kato}).  The classical definition due to Bloch-Kato (cf. \cite{bk}) is
to take this cohomology to be
$$\varprojlim_n E(K_v)/l^n.$$
However, this doesn't work for our purpose since this group might not have
finite projective dimension over $\Z_l[G]$.

So we consider the short exact sequence
$$ 0 \rightarrow E_1(K_v) \rightarrow E_0(K_v) \rightarrow E(k_v) \rightarrow
0,$$
where $E_0(K_v)$ is the group of nonsingular points, $E_1(K_v)$ is the
subgroup of points that are trivial modulo the maximal ideal $\mathfrak{m}_v$
and $k_v$ is the residue field.  One has an isomorphism between $E_1(K_v)$
and the formal group $\widehat{E}(\mathfrak{m}_v)$.  We let $n_v$ to be the
least positive integer such that $\widehat{E}(\mathfrak{m}_v^{n_v}) \simeq
\mathfrak{m}_v^{n_v}$ (see \cite{silverman} for details) and such that
the ramification index $e(K_v/\Q_l)$ divides $n_v$.

We now define for $v | l$
$$H_f^1(K_v,T_l(E)) := im(\widehat{E}(\mathfrak{m}_v^{n_v})),$$
where the right hand side is the image of the $l$-adic completion of the formal
group in $\varprojlim_n E(K_v)/l^n$.

\noindent {\bf Remark.}  Note that for a tamely ramified extension $K_v/\Q_l$
with ramification index $e := e(K_v/\Q_l)$, and for a positive integer $r$,
one has that
$\mathfrak{m}_v^{er} = l^r \mathcal{O}_{K_v}
\simeq \mathcal{O}_{K_v} \simeq \Z_l[\mathrm{Gal}(K_v/\Q_l)]$, as
$\Z_l[\mathrm{Gal}(K_v/\Q_l)]$-modules.  Therefore,
the above definition of $R\Gamma_f$ ensures that $R\Gamma(K_l,T_l(E))$ is
a free $\Z_l[G]$-module of rank 1.

Thus we have the following.
\begin{lemma} \label{finite_definition}
The below definitions of $R\Gamma_f(K_v,T_l(E))$ for $v \in S_l(K)$ satisfy
the conditions that
\begin{itemize}
\item there exists a map $R\Gamma_f(K_v,T_l(E)) \rightarrow R\Gamma(K_v,V_l(E))$
of complexes which upon tensoring with $\Q_l$ becomes quasi-isomorphic to
$R\Gamma_f(K_v,V_l(E)) \rightarrow R\Gamma(K_v,V_l(E))$;
\item $R\Gamma_f(K_p,T_l(E))$ is a perfect $\Z_l[G]$-complex.
\end{itemize}
\begin{enumerate}
\item[(i)] If $v \in S_{\infty}(K)$ then $R\Gamma_f(K_v,T_l(E)) :=
H^0(K_v,T_l(E))$ (placed in degree 0).
\item[(ii)] If $v \in S_f(K)$ and $l$ does not divide the order of the image
of $I_v$ in $G$ then define $R\Gamma_f(K_v,T_l(E))$ to be the complex
$$ T_l(E)^{I_v} \stackrel{1-Fr^{-1}_v}{\longrightarrow} T_l(E)^{I_v}.$$
\item[(iii)] If $v \in S_f(K)$ and $l$ divides the order of the image of
$I_v$ in $G$ then define $R\Gamma_f(K_v,T_l(E))$ to be the zero complex.
\item[(iv)] If $v \mid l$ then $R\Gamma_f(K_v,T_l(E)) :=
\mathrm{im}(\widehat{E}(\mathfrak{m}_v^{n_v}))$, where $\mathfrak{m}_v$ is the
maximal ideal in $\mathcal{O}_{K_v}$ and $n_v$ is smallest positive integer
divisible by the ramification index $e(K_v/\Q_l)$ and such that
$\widehat{E}(\mathfrak{m}_v^{n_v}) \simeq \mathfrak{m}_v^{n_v}$.
\end{enumerate}
\end{lemma}

\begin{remarks}
\begin{enumerate}
\item For $v \in S_f(K)$, the above definition of $R\Gamma_f(K_v,T_l(E))$
differs from the classical definition by a factor of $|c_vL_v(E,1)^{-1}|_l$
or $|c_v|_l$ depending on whether or not $l$ divides the order of the image
of $I_v$ in $G$.  Here $c_v$ is the order of the group of components.
\item For $v|l$ the above definition differs by the classical definition by
$$ [\varprojlim_n E(K_v)/l^n : H_f^1(K_v,T_l)] = |c_v|E(k_v)||_l \cdot
|k_v|^{n_v - 1} = |c_vL_v(E,1)^{-1}|_l \cdot |k_v|^{n_v},$$
where $k_v$ is the residue field at $v$.
\item For $v \mid l$, note that the following diagram is commutative:
\begin{displaymath}
\xymatrix{
E(K_v)\otimes_{\Z_l} \Q_l & Tan(E_{K_v}) \ar[l]_{exp} \\
\widehat{E}(\mathfrak{m}_v^{n_v}) \otimes_{\Z_l} \Q_l \ar[u] \ar[r]^{log} &
K_v \ar[u]^{\cdot \omega_0^*}
}
\end{displaymath}
We shall therefore compute the image of $\oplus_{v|l} H_f^1(K_v,T_l(E))$
under the comparison isomorphism via the logarithm map.
\end{enumerate}
\end{remarks}

We can now define the global finite complex as in \cite{bf}.  To be precise,
we let
$$ R\Gamma_{/f}(K_v,T_l(E)) := \mathrm{Cone}(R\Gamma_f(K_v,T_l(E))
\rightarrow R\Gamma(K_v,T_l(E))) $$
and
$$ R\Gamma_f(K,T_l(E)) := \mathrm{Cone}(R\Gamma(\mathcal{O}_{K,S_l},T_l(E))
\rightarrow \oplus_{v \in S} R\Gamma_{/f}(K_v,T_l(E)))[-1].$$
The cohomology of $R\Gamma_f(K,T_l(E))$ is given by the following
lemma.

\begin{lemma} \label{global_f}
If $\Sha(E(K))$ is finite then
\begin{eqnarray*}
H_f^0(K,T_l(E)) = 0, & H^3(K,T_l(E)) \simeq
\mathrm{Hom}_{\Z} (E(K)_{l^{\infty}}, \Q_l/\Z_l),
\end{eqnarray*}
and there are exact sequences of $\Z_l[G]$-modules
$$ 0 \rightarrow H_f^1(K,T_l(E)) \rightarrow E(K)\otimes_{\Z} \Z_l
\rightarrow \oplus_{v \in S_{l,f}(K)} \Phi_v $$
$$
\rightarrow H_f^2(K,T_l(E)) \rightarrow H_f^2(K,T_l(E))_{BK} \rightarrow 0,$$
$$ 0 \rightarrow \Sha(E(K))_{l^{\infty}} \rightarrow H_f^2(K,T_l(E))_{BK}
\rightarrow \mathrm{Hom}_{\Z}(E(K), \Z_l) \rightarrow 0,$$
where,
\begin{displaymath} \Phi_v :=
\left( \varprojlim_n \frac{E(K_v)}{l^n} \right)/H^1_f(K_v,T_l(E)).
\end{displaymath}
\end{lemma}

\begin{proof}
By the above definitions of the finite complexes one has a commutative
diagram
\begin{displaymath}
\xymatrix{
R\Gamma_c(\mathcal{O}_{K,S_l(K)}, T_l(E)) \ar[r] \ar@{=}[d] &
R\Gamma_f(K, T_l(E)) \ar[r] \ar[d] & \oplus_{v \in S}
R\Gamma_f(K_v, T_l(E)) \ar[d] \\
R\Gamma_c(\mathcal{O}_{K,S_l(K)}, T_l(E)) \ar[r] \ar[d] &
R\Gamma_f(K, T_l(E))_{BK} \ar[r] \ar@{-->}[d] & \oplus_{v \in S}
R\Gamma_f(K_v, T_l(E))_{BK} \ar@{-->}[d] \\
0 \ar[r] & \oplus_{v \in S_{l,f}(K)} W_v \ar@{=}[r] & \oplus_{v \in
S_{l,f}(K)} W_v }
\end{displaymath}
with rows being distinguished triangles.  It follows from the octahedral
axiom that there exists complexes $W_v$ such that the columns in the above
diagram are also distinguished.  We note that $\Phi_v = H^1(W_v)$.  The lemma
now follows from the results of \cite{bk} on the classical finite cohomology
of $R\Gamma_f(K,T_l(E))_{BK}$ and from the long exact sequences arising from
the distinguished columns.
\end{proof}

\begin{corollary} \label{global_finite}
If $E(K)$ is torsion and $\Sha(E(K))$ is finite, then one has
$$ H^0_f(K,T_l(E)) = H^1_f(K,T_l(E)) = 0, H^3_f(K,T_l(E)) \simeq
\mathrm{Hom}_{\Z} (E(K)_{l^{\infty}},\Q_l/\Z_l),$$
and the following sequence is exact:
$$ 0 \rightarrow E(K)_{l^{\infty}} \rightarrow \oplus_{v \in S_{l,f}(K)} \Phi_v
\rightarrow H^2_f(K,T_l(E)) \rightarrow \Sha(E(K))_{l^{\infty}} \rightarrow
0.$$
\end{corollary}

\section{Special values of $L$-functions}

\subsection{Modular symbols} \label{modular_symbols}

Let $\mathcal{H} = \{z \in \C | Im(z) > 0 \}$ denote the upper half plane
and $\mathcal{H}^* = \mathcal{H} \cup \Q \cup \{ \infty \}$ denote the
extended upper half plane.  Let $\Gamma$ be a congruence subgroup and let
$X_{\Gamma} : = \Gamma\backslash \mathcal{H}^*$ be the corresponding modular
curve.  For cusps $\alpha, \beta \in \Q \cup \{\infty\}$ consider a smooth
path in $\mathcal{H}^*$ from $\alpha$ to $\beta$.  Let $\{ \alpha, \beta
\}_{\Gamma}$ denote the image of the path in $X_{\Gamma}$.  The Manin-Drinfeld
theorem says that for any cusps $\alpha, \beta \in \Q \cup \{\infty\}$ one
has $\{ \alpha, \beta \}_{\Gamma} \in H_1(X_{\Gamma},\Q)$.  For a cusp form
$f \in S_2(\Gamma)$ we let
$$ \left< \{\alpha, \beta\}_{\Gamma}, f \right> : = 2 \pi i
\int_{\alpha}^{\beta} f(z) dz.$$
Note that the element $\{ \alpha, \beta \}_{\Gamma} \in H_1(X_{\Gamma},\Q)$
and the above integral are both independent of the path chosen.

Using Mellin transformation one can deduce the following:

\begin{proposition}
Let $N$ be a prime, $\Gamma = \Gamma_0(N)$, $f$ be a new form and $\eta$ be
a Dirichlet character of prime conductor $l$ with $(l,N) = 1$.  Then
$$ L(f,1) = - \left< \{0,\infty \}_{\Gamma},f \right>,$$
and
$$ L(f\otimes \eta,1) = \frac{g(\eta)}{l} \sum_{a=1}^l \bar{\eta}(a) \left<
\{0,a/l\}_{\Gamma},f \right>,$$
where $g(\eta) = \sum_{n=1}^l \eta(n) e^{2\pi i n/l}$ is the Gauss sum.
\end{proposition}
See \cite{cremona} for a proof of the above proposition.

Now consider the well-known pairing $ H_1(X_{\Gamma},\C)^+ \times S_2(\Gamma)
\rightarrow \C$ given by
$$ \left< \gamma, f \right> = \int_{\gamma} f(z)dz.$$  This gives an
isomorphism $S_2(\Gamma) \simeq H_1(X_{\Gamma},\C)^{+*}$ of $\C$-vectorspaces.
Given a rational new form $f \in S_2({\Gamma})$, it corresponds to an element
$\{\alpha, \beta\} \in H_1(X_{\Gamma},\Q)^+$ via this isomorphism.  Note that
this element is unique up to sign if we further restrict it to be an element
of $H_1(X_{\Gamma},\Z)^+$.  We shall denote this 1-cycle by $\gamma_f$.  Since
the pairing $\left<,\right>$ is compatible with the action of the Hecke
operators, it follows that $\gamma_f$ is a common eigen
vector for all the Hecke operators with eigenvalues same as that of the
new form.  Thus, by looking at the Hecke action on $H_1(X_{\Gamma},\Z)^+$ we
can compute $\gamma_f$.  Let $V$ be the $\Q$-vector space generated by
$\gamma_f$.  Note that we can construct the complementary subspace $V'$ on
which the pairing $\left<\cdot,f\right>$ is trivial.  Then, for any $\gamma
\in H_1(X_{\Gamma},\Q)^+$ we have
$$ \int_{\gamma} f = \int_{\gamma|_V} f $$
where $\gamma|_V$ is the projection of $\gamma$ onto the subspace $V$.  In
fact, $\gamma|_V$ is a rational multiple of $\gamma_f$.  So, if we let
$\Omega_f := \int_{\gamma_f} f$ then we see that $\int_{\gamma} f$ is a
rational multiple of $\Omega_f$.  The following well-known proposition
gives the equality between the periods of an elliptic curve and the
corresponding modular form.
 
\begin{proposition} \label{ell_to_modular}
Let $E$ be a strong Weil curve defined over $\Q$ and let $f$ be the
normalized rational
new form corresponding to $E$.  Then, the real period associated to $E$
equals $c_Ed\Omega_f$, where $d$ is the number of components in the real locus
of $E$ (that is, $d = 2$ if the corresponding lattice is rectangular, $d=1$
otherwise) and $c_E$ is the Manin constant.
\end{proposition}
 
Note that, $\{0,a/l\} + \{0,(l-a)/l\} \in H_1(X_{\Gamma},\Z)^+$ and
$\{0,\infty\} \in H_1(X_{\Gamma},\Z)^+$.  Hence we can compute the
$L$-values of a rational new
form and its twists using the above computation.  Implementation of such
a computation is studied in detail by Cremona (see \cite{cremona}).

\noindent {\bf Remarks.}
\begin{enumerate}
\item The Manin constant $c_E$ is known to be trivial for
elliptic curves $E$ of prime conductor (see \cite{au}).
\item
An analogous result can be obtained for the imaginary period of
the elliptic curve by looking at $H_1(X_{\Gamma},*)^{-}$, the eigenspace
for the action of the complex conjugation with eigenvalue $-1$.
\end{enumerate}

\subsection{Nonabelian twists} \label{nonabelian_twists}

Let $E$ be an elliptic curve defined over $\Q$.  Let $\tau$ be an
(irreducible) self-dual representation of $G = \mathrm{Gal}(K/\Q)$ of
dimension $d$ and let $N(E,\tau)$ be the conductor of $E \otimes \tau$.  The
$L$-function $L(E\otimes \tau, s)$ has a meromorphic continuation to the
whole $s$-plane and it satisfies a functional equation
\begin{eqnarray} \label{functional_equation}
\widehat{L}(E\otimes \tau,s) = \pm \widehat{L}(E\otimes \tau,2-s)
\end{eqnarray}
where
$$ \widehat{L}(E \otimes \tau, s) = \left( \frac{\sqrt{N(E,\tau)}}{\pi^d}
\right)^s \Gamma \left( \frac{s}{2}\right)^d \Gamma
\left(\frac{s+1}{2}\right)^d L(E \otimes \tau,s ).$$
Let $A = A(E,\tau) := \frac{\sqrt{N(E,\tau)}}{\pi^d}$ and $\gamma(s) =
\Gamma(\frac{s}{2})^d \Gamma(\frac{s+1}{2})^d$.

Let $\phi(s)$ be the inverse Mellin transform of $\gamma(s)$, that is,
$$ \gamma(s) = \int_0^{\infty} \phi(t)t^s \frac{dt}{t}.$$
Let
$$ G_s(t) = t^{-s} \int_{t}^{\infty} \phi(x)x^s \frac{dx}{x} $$
be the incomplete Mellin transform of $\phi(t)$.  Then one has
\begin{proposition} \label{approx}
\begin{eqnarray} \label{lseries}
\widehat{L}(E \otimes \tau,s) = \sum_{n=1}^{\infty} a_n G_s \left(
\frac{n}{A} \right) \pm \sum_{n=1}^{\infty} a_n G_{2-s} \left( \frac{n}{A}
\right).
\end{eqnarray}
\end{proposition}

\begin{proof}
See \cite{td} or \cite{tollis}.
\end{proof}

For fixed $s$, the series (\ref{lseries}) converges exponentially with $t$
and therefore we can use this series to get numerical approximations of the
value $L(E\otimes \tau,1)$.  The rate of convergence of the series depends
on the conductor $N(E \otimes \tau)$.  We roughly need to sum
$\sqrt{N(E \otimes \tau)}$
terms in the series to obtain an approximation.  Note that if the bad primes
of $E$ and $\tau$ do not intersect, then the conductor of $E \otimes \tau$ is
$N(E, \tau) = N(E)^d N(\tau)^2$.  Therefore, obtaining numerical
approximations to the value $L(E\otimes \tau,1)$ is computationally infeasible
for large field extensions.

In \cite{td} Dokchitser has explained how to compute $G_s(t)$ efficiently and
this has been implemented in \cite{computel}.  Our numerical approximations
use this particular implementation.

\noindent {\bf Remark.}  The Proposition \ref{approx} is in fact a special
case of a more general result that holds for any $L$-series having a
meromorphic continuation and satisfying a functional equation of type
(\ref{functional_equation}).

\section{An example}
\subsection{The group $S_3$.}  Let $r$
and $s$ denote elements of $S_3$ of order 2 and 3 respectively.  Let
$C_3$ denote the subgroup of $S_3$ of order 3.  We denote by $\chi_0,
\chi$ and $\psi$ the trivial, nontrivial abelian and the nonabelian
characters of $S_3$ respectively.  We define $\rho : S_3 \rightarrow
GL_2(\C)$ by
$$ \rho(r) = \left( \begin{array}{rr} 0 & -1 \\ -1 & 0 \end{array} \right),
\rho(s) = \left( \begin{array}{rr} -1 & -1 \\ 1 & 0 \end{array} \right).  $$
This is an irreducible representation of $S_3$ whose character is $\psi$.

As before, for a character $\eta \in \widehat{S_3}$, we let $e_{\eta} =
\sum_{g \in S_3} \eta(g)g^{-1}$.  Further, we fix indecomposible
idempotents $e_{\psi,1}$ corresponding to the character $\psi$.  Let $\{
e_{\psi,1}, e_{\psi,2}\}$ be a basis of $e_{\psi,1}\Q_l[S_3]$.
To be precise, we have
\begin{eqnarray*}
e_{\chi_0} & = & (1+s+s^2+r+rs+rs^2)/6,\\
e_{\chi} & = & (1+s+s^2-r-rs-rs^2)/6,\\
e_{\psi} & = & (2-s-s^2)/3,\\
e_{\psi,1} & = & (1+rs^2-s^2-rs)/3,\\
e_{\psi,2} & = & (s+rs-s-r)/3.\\
\end{eqnarray*}

For any finite $\Z_l[S_3]$-module $N$ of finite projective
dimension, there exists a short exact sequence
$$ 0 \rightarrow P_1 \rightarrow P_0 \rightarrow N \rightarrow 0,$$
of $\Z_l[S_3]$-modules, where $P_i$'s are projective over $\Z_l[S_3]$.
Let $\tau: P_0 \otimes \Q_l \simeq P_1 \otimes \Q_l$ be the induced map.
Since the class group $\mathrm{Cl}(\Z_l[S_3])$ is trivial (cf. \cite{cr},
49.11), it follows
that $P_0 \simeq P_1$ as $\Z_l[S_3]$-modules.  Picking such an isomorphism,
$\tau$ gives an automorphism of $P_0 \otimes \Q_l$.  This defines an element
of $\zeta(\Q_l[S_3])^{\times}$ whose image under
$$\widehat{\delta} : \zeta(\Q_l[S_3])^{\times} \rightarrow K_0(\Z_l[S_3];\Q_l)
$$
is $(P_1,P_0; \tau) \in K_0(\Z_l[S_3];\Q_l)$.  We shall denote this element
by $\epsilon(N)$.  Note that by choosing a different presentation for $N$
and by choosing a different isomorphism between the modules in the
presentation, $\epsilon(N)$ changes by an element of
$K_1(\Z_l[S_3])$.  Thus, $\epsilon(N)$ is well-defined upto an element of
$K_1(\Z_l[S_3])$.  For $a, b \in \zeta(\Q_l[S_3])^{\times}$ we shall write
$a \thicksim b$ if $a/b$ is in the image of $K_1(\Z_l[S_3])$.  If $N$ is
the trivial module, then we have $\epsilon(N) \thicksim 1$.
If $C_{\bullet}$ is a perfect complex of $\Z_l[S_3]$-modules with finite
cohomology then, by Remark 2.2 of \cite{flach1}, it corresponds to an
element $\overline{\epsilon}(N)$ of $K_0(\Z_l[S_3];\Q_l)$.  We let
$\epsilon(C_{\bullet})$ to be an element of $\zeta(\Q_l[S_3])^{\times}$
whose image in $K_0(\Z_l[S_3];\Q_l)$ equals $\overline{\epsilon}(N)$.  Note
that if $C_{\bullet} = N[i]$, a complex concentrated in a single degree,
then $\epsilon(C_{\bullet}) \thicksim \epsilon(N)^{(-1)^i}$.

For $q \in S_f$, if $\Phi_q$ is of finite projective dimension, then we fix
a $\Z_l[S_3]$-presentation for $\Phi_q$ and also
fix isomorphisms between the modules appearing in the presentation.  This
fixes $\epsilon(\Phi_q) \in \zeta(\Q_l[S_3])^{\times}$.

\subsection{Setting.}  Let $K$ be the splitting field of $p(x) = x^3 - 4x
+ 1$.  Let $E$ be the elliptic curve $y^2 +  y = x^3 - x^2 - 10x - 20$.
This is the curve 11A1 in the sense of Cremona.  Then we have the following.
\begin{enumerate}
\item $K/\Q$ is a real $S_3$-extension since $p(x)$ is an irreducible
polynomial with discriminant $229$.
\item The conductor of $E$ is 11, and the discriminant of $K$ is $229^3$.
Since the conductor of $E$ is a prime, the corresponding Manin constant
$c_E = 1$.  Also, the real locus of $E$ has only one component and therefore
the real period associated to $E$ equals $\Omega_f$ where $f$ is the
normalized rational new form corresponding to $E$.
\item The set $S = \{ 11, 229, \infty \}$.
The residual indices of $11$ and $229$ are 3 and 1 respectively.
We note that $|E_{\mathrm{ns}}(k_v)| = 1330$, for $v \mid 11$, and
$|E_{\mathrm{ns}}(k_v)| = 215$, for $v \mid 229$.
\item  The above shows that $|E(K)_{\mathrm{tors}}| \leq 5$ (see
\cite{silverman} for details).  But, from the tables in \cite{cremona}
we know that $|E(\Q)| = 5$.  Therefore we have $|E(K)_{\mathrm{tors}}| = 5$.
\item $L(E/K,1) \neq 0$.  This follows from the value computed below.
Thus, by a theorem of Zhang (cf. \cite{zhang}), we have that $E(K)$ is finite.
\item $E$ has split multiplicative reduction at $p = 11$.  The Tamagawa
factor $c_{11}$ is 5 and therefore, $c_v = 125$ for all $v \mid 11$.
\item Suppose that $|\Sha(E/K)_{2^{\infty}}| = |\Sha(E/K)_{3^{\infty}}|
= 1$.  Then, for any $l$, we have that $\Phi_q$ is of finite projective
dimension for all $q \in S_f$.
\end{enumerate}

\subsection{Arithmetic values}
\subsubsection{Global finite cohomologies.}
We henceforth assume that $\Sha(E/K)$
is trivial.  So, by Corollary \ref{global_finite}, the global finite
cohomologies are concentrated in degree two and three, and are given by
$$H^3_f(K,T_l(E)) \simeq \mathrm{Hom}(E(K)_{l^{\infty}}, \Q_l/\Z_l),$$
and,
$$ 0 \rightarrow E(K)_{l^{\infty}} \rightarrow \oplus_{v \in S_{l,f}(K)}
\Phi_v \rightarrow H^2_f(K,T_l(E)) \rightarrow 0.$$  Note that, in
our setting, for $l \neq 5$, $H^3_f$ vanishes and $H^2_f$ is isomorphic
to $\oplus_{v \in S_{l,f}(K)} \Phi_v$.  And for $l = 5$, we have
$$ E(K)_{l^{\infty}} \simeq \mathrm{Hom}(E(K)_{l^{\infty}},\Q_l/\Z_l)
\simeq \Z/l\Z,$$
as $\Z_l[S_3]$-modules.  So, defining $u_l := \epsilon(\Z/l\Z)^2$ for
$l = 5$ and $u_l := 1$ for $l \neq 5$, we have
\begin{eqnarray} \label{global_finite_e}
\epsilon(R\Gamma_f(K,T_l(E))) \thicksim
\epsilon(\oplus_{v \in S_{l,f}(K)} \Phi_v) \cdot u_l^{-1}.
\end{eqnarray}

\subsubsection{Local finite cohomologies.}
For $q \in S_f$, the complex $R\Gamma_f(K_q,T_l(E))$ has finite cohmologies.
Therefore, it's contribution to $\Xi \otimes \Q_l$ can be measured by
$\epsilon(R\Gamma_f(K_q,T_l(E)))$.  This is trivial if the complex is the
zero complex.  Otherwise, it equals
$\left( \sum_{\eta \in \widehat{G}} |L_q(E \otimes \eta, 1)|_l\right)^{-1}$.

For $v \mid l$, we have that $H_f^1(K_v,T_l(E)) =
\widehat{E}(\mathfrak{m}_v^{n_l})$, where $n_l$ is trivial for $l > 3$,
$n_l = 2$ for $l = 3$, and $n_l = 3$ for $l = 2$.  Define $\tilde{\Phi}_v$
by the exact sequence
$$ 0 \rightarrow \mathfrak{m}_v/\mathfrak{m}_v^{n_l} \rightarrow \Phi_v
\rightarrow \tilde{\Phi}_v\rightarrow 0.$$
So, $\tilde{\Phi}_v$ is trivial if $l =2, 3$, and $\tilde{\Phi}_v = \Phi_v$
otherwise.
Next, we note that via the exponential map, we have an exact sequence
$$ 0 \rightarrow \oplus_{v \mid l} H^1_f(K_v,T_l(E)) \rightarrow
\oplus_{v \mid l} \mathcal{O}_{K_v} \cdot \omega_0 \rightarrow
\oplus_{v \mid l} \mathcal{O}_{K_v}/\mathfrak{m}_v^{n_l} \rightarrow 0. $$
Also, we have an exact sequence
$$ 0 \rightarrow \oplus_{v \mid l} \mathfrak{m}_v/\mathfrak{m}_v^{n_l}
\rightarrow \oplus_{v \mid l}
\mathcal{O}_{K_v}/\mathfrak{m}_v^{n_l} \rightarrow \oplus_{v \mid l}
k_v \rightarrow 0.$$
From the above three exact sequences we get the following identity in
$V(\Z_l[G])$:
\begin{eqnarray} \label{finite_factor}
 [\oplus_{v \mid l} H^1_f(K_v,T_l(E))] \oplus [\oplus_{v \mid l} \Phi_v]
= [ \oplus_{v \mid l} \mathcal{O}_{K_v} \cdot \omega_0] \oplus
[\oplus_{v\mid l} k_v]^{-1} \oplus [\oplus_{v \mid l} \tilde{\Phi}_v].
\end{eqnarray}

Recall also that for $v \mid \infty$, we have $H^1_f(K_v,T_l(E)) \simeq
T_l^{S_{3,v}}[0]$.  As before, we choose $\Z_l[S_3]$-generators
$\alpha_0 \cdot \omega_0$ and $\tilde{\gamma}_1$ for $\mathcal{O}_{K_v}
\cdot \omega_0$ and $T_l^{S_{3,v}}[0]$ respectively.

\subsubsection{Generator $\mu_l$.}
The above choice of generators for $\mathcal{O}_{K_v} \cdot \omega_0$ and
$T_l^{S_{3,v}}[0]$ define a $\zeta(\Q_l[S_3])$-generator $\mu_l$ of
$\Det_{\Q_l[S_3]} R\Gamma_c(\mathcal{O}_{K,S_l},V_l(E))$.  Note that the
map $\vartheta_l : \Xi \otimes \Q_l \simeq 
\Det_{\Q_l[S_3]} R\Gamma_c(\mathcal{O}_{K,S_l(K)},V_l(E))$ is given via the
quasi-isomorphism
\begin{eqnarray} \label{triv_1}
R\Gamma_f(K_v,V_l(E)) \rightarrow \left( V_l(E)^{I_v}
\stackrel{1-Fr^{-1}_v}{\longrightarrow} V_l(E)^{I_v} \right)
\end{eqnarray}
for $v \in S_f(K)$, and via the distinguished triangle
\begin{eqnarray} \label{triv_2}
\left( H_{dR}(M)/F^0 \otimes_{\Q} \Q_l \right)[0] & \rightarrow & \oplus_{v
\mid l} R\Gamma_f(K_v,V_l(E)) \\ & \rightarrow & \left( \oplus_{v \mid l}
D_{cris}(V_l(E)) \stackrel{1-Fr^{-1}_v}{\longrightarrow} \oplus_{v \mid l}
D_{cris}(V_l(E)) \right). \nonumber
\end{eqnarray}
From (\ref{global_finite_e})-(\ref{triv_2}), the inverse image
$\vartheta_l^{-1}(\mu_l)$ in $\Xi \otimes \Q_l$ is given by
$\vartheta_l^{-1}(\mu_l) = \xi_l \cdot \tilde{\gamma}_1 \otimes
(\omega_0 \otimes \alpha_0)^{*}$ where
\begin{eqnarray} \label{vl_inverse}
\xi_l & := & \epsilon(\oplus_{v \mid l} k_v)^{-1} \cdot u_l^{-1}
\cdot \epsilon(\tilde{\Phi}_l)
\cdot \prod_{q \in S_f} \epsilon(\Phi_q) \\
& & \cdot \prod_{q \in S_f} \epsilon( R\Gamma_f(K_q,T_l(E)) )^{-1}
\cdot \prod_{q \in S_{f,l}}
\left( \sum_{\eta \in \widehat{S_3}} L_q(E \otimes \eta,1)e_{\eta}
\right)^{-1}. \nonumber
\end{eqnarray}
The first term is the contribution of $\oplus_{v \mid l} k_v$ from
(\ref{finite_factor}).  This term equals $(le_{\chi_0} + le_{\chi}
+l^2e_{\psi})$ for $l \neq 229$, and equals $(le_{\chi_0} + e_{\chi}
+ le_{\psi})$ for $l = 229$.

\subsection{$L$-values and equivariant conjecture.}
We now prove Theorem \ref{intro_theorem}.
\begin{proof}[Proof of Theorem \ref{intro_theorem}]
Let $(E,K)$ be as above.  Using the methods prescribed in Sections
\ref{modular_symbols} and \ref{nonabelian_twists} we compute the
following special values of the twisted $L$-functions.
\begin{eqnarray*}
L(E,1) & = & \Omega/5, \\
L(E\otimes \chi,1) & = & 5\Omega/\sqrt{229}, \\
L(E \otimes \psi,1) & \approx & 25\Omega^2/\sqrt{229}. \\
\end{eqnarray*}

On the other hand, fixing an $\alpha_0$ we get
\begin{eqnarray} \label{period_values}
\det \left( \sum_{g \in S_3} \Omega^{-1} g(\alpha_0) \rho_{\eta}(g^{-1})
\right) & = & \left\{
\begin{array}{ll}
\Omega^{-1} & \textrm{if } \eta = \chi_0, \\
\Omega^{-1} \cdot \sqrt{229} & \textrm{if } \eta = \chi, \\
\Omega^{-2} \cdot \sqrt{229} & \textrm{if } \eta = \psi.
\end{array}
\right.
\end{eqnarray}
These computations were carried out using PARI/GP (cf. \cite{pari}).
Thus, from Proposition \ref{period_iso}, we get
\begin{eqnarray} \label{l_image}
\vartheta_{\infty}^{-1}(L^*(h^1(E_K)(1))^{-1}) = (e_{\chi_0}/5 + 5e_{\chi}
+ 25e_{\psi}) (\tilde{\gamma} \otimes (\omega_0 \otimes \alpha_0)^*).
\end{eqnarray}
This verifies the rationality conjecture.

The local $L$-values $L_p(E \otimes \eta,1)^{-1}$ are given by the following
table.
\begin{table}[!ht] \caption{Local $L$-values}
\centering
\begin{tabular}{|c|c|c|c|}
\hline
$p$ & $L_p(\chi_0,1)$ & $L_p(\chi,1)$ & $L_p(\psi,1)$ \\ \hline
2 & $5/2$ & $1/2$ & $5/4$ \\
3 & $5/3$ & $5/3$ & $4/9$ \\
5 & 1 & 1 & $28/25$ \\
11 & $10/11$ & $10/11$ & $133^2/11^4$ \\
229 & $215/229$ & $1$ & $215/229$ \\
\hline
\end{tabular}
\end{table}

Recall that there is an isomorphism $\tau: \Xi \rightarrow \zeta(\Q[S_3])$
which when tensored with $\R$ gives $-\mathrm{nr}^{-1}(L^*(h^1(E_K)(1)))
\cdot \vartheta_{\infty}$.  This gives us an isomorphims $\tau_l: \Xi
\otimes \Q_l \rightarrow \zeta(\Q_l[S_3])$.  Combining this with
$\vartheta_l^{-1}$, we get
$$ \beta_l :\Det_{\Q_l[S_3]} R\Gamma_c(\mathcal{O}_{K,S_l},V_l(E)) \rightarrow
\zeta(\Q_l[S_3]).$$
From (\ref{vl_inverse}) and (\ref{l_image}), we have
\begin{eqnarray} \label{beta_value}
\beta_l(\mu_l) =  (e_{\chi_0}/5 + 5e_{\chi} + 25e_{\psi}) \cdot \xi_l.
\end{eqnarray}
Our aim is to show that this is an element of $K_1(\Z_l[S_3])$ and this
verifies the integrality part of the conjecture.  We split the verification
into various cases.

\subsubsection{Case $l = 2$.}
In this case, we first note that $\Phi_v$ is trivial for all $v \in S_f(K)$.
Also, $\tilde{\Phi}_2$ is trivial and $u_2 = 1$.  So, from (\ref{vl_inverse})
we have
$$ \xi_l = (2e_{\chi_0} + 2e_{\chi} + 4e_{\psi}) \cdot \left( \sum_{\eta \in
S_3} |L_{11}(E \otimes \eta, 1)|_2 e_{\eta} \right) \cdot \prod_{q \in
\{2,11,229\}} \left( \sum_{\eta \in S_3} L_q(E \otimes \eta, 1)e_{\eta}
\right). $$
So from (\ref{beta_value}) we have
\begin{eqnarray*}
\beta_2(\mu_2) & = & (a_{\chi_0} e_{\chi_0} + a_{\chi} e_{\chi} +
a_{\psi} e_{\psi}),
\end{eqnarray*}
where $a_{\eta}$'s are units in $\Z_2$.  Therefore it follows that
$\beta_2(\mu_2), \beta_2(\mu_2)^{-1} \in \Z_2[S_3]$.  Hence,
multiplication by $\beta_2(\mu_2)$ defines an isomorphim from between
free rank-1 $\Z_2[S_3]$-modules.  This shows that $\beta_2(\mu_2)$ is in
the image of $K_1(\Z_2[S_3])$.  This verifies the equivariant conjecture
for $l =2$.

\subsubsection{Case $l = 3$.}
In this case, the modules $\Phi_v$ are all trivial.  Also, $\tilde{\Phi}_3$
is trivial and $u_3 = 1$.  Further, the
cohomologies of $R\Gamma_f(K_q,T_l(E))$ are also trivial for $q = 11, 229$.
Thus (\ref{beta_value}) reduces to
\begin{eqnarray*}
\beta_3(\mu_3) & = & (a_{\chi_0} e_{\chi_0} + a_{\chi} e_{\chi} + a_{\psi}
e_{\psi}),
\end{eqnarray*}
where $a_{\chi_0} \equiv -a_{\chi} \equiv -a_{\psi} \equiv 1 \pmod{\Z_3}$.

Now, multiplication by $r$  defines an isomorphim from between free
rank-1 $\Z_3[S_3]$-modules.  Thus, $r$ defines an element $\theta \in
K_1(\Z_3[S_3])$.  The image of $\theta$ in $\zeta(\Q_3[S_3])^{\times}$
is $\overline{\theta} := (e_{\chi_0} - e_{\chi} - e_{\psi})$.  Consider
the element $\overline{\theta} \cdot \beta_l(\mu_l) \in
\zeta(\Q_3[S_3])^{\times}$.  All the coefficients of $e_{\eta}$ in
$\overline{\theta} \cdot \beta_l(\mu_l)$ are in $1 + 3\Z_3$.  It therefore
follows that $\overline{\theta} \cdot \beta_l(\mu_l)$ and its inverse are
both elements of $\Z_3[S_3]$, and thus they are in the image of
$K_1(\Z_3[S_3])$.  Since $\overline{\theta}$ is in the image of
$K_1(\Z_3[S_3])$, so is $\beta_l(\mu_l)$.  This verifies the equivariant
conjecture for $l =3$.

\subsubsection{Case $l = 5$.}  First,  consider the sequence
$$ 0 \rightarrow \Z_5[S_3] \stackrel{\nu}{\rightarrow} \Z_5[S_3]
\rightarrow \Z/5\Z \rightarrow 0,$$
where $\nu$ is multiplication by $\theta = 5e_{\chi_0} + e_{\chi} +
e_{\psi}$.  Note that for any $g \in S_3$, $(1-g)\cdot \theta = (1-g)$
and $(e_{\chi_0} + 5e_{\chi} + 5e_{\psi}) \cdot \theta = 5$, and therefore
it follows that the above sequence is exact.  Thus, we have $u_5 =
25e_{\chi_0} + e_{\chi} + e_{\psi}$.

The contributions of $\Phi_q$'s are given by the following.
\begin{lemma}
For $l =5$ we have
\begin{eqnarray*}
\epsilon(\Phi_{q}) & \thicksim & \left\{
\begin{array}{ll}
5e_{\chi_0} + 5e_{\chi} + e_{\psi} & \textrm{ if } q = 5, \\
25e_{\chi_0} + 25e_{\chi} + 25e_{\psi} & \textrm{ if } q = 11, \\
1 & \textrm{ if } q = 229.
\end{array}
\right.
\end{eqnarray*}
\end{lemma}
\begin{proof}
We have that $\Phi_5 \simeq \mathrm{Ind}_{\Z_l[C_3]}^{\Z_l[S_3]} \Z/5\Z$
with trivial action of $C_3$ on $\Z/5\Z$.  Consider the sequence
$$ 0 \rightarrow \Z_5[S_3] \stackrel{\nu}{\rightarrow} \Z_5[S_3]
\rightarrow \Phi_5 \rightarrow 0,$$
where $\nu$ is multiplication by $\theta = 5e_{\chi_0} + 5e_{\chi} +
e_{\psi}$.  Note that $(1-s)\cdot \theta = (1-s)$ and $(e_{\chi_0} +
e_{\chi} + 5e_{\psi}) \cdot \theta = 5$.  Therefore it follows that
the above sequence is exact.  Thus, we have $\epsilon(\Phi_5) =
5e_{\chi_0} + 5e_{\chi} + e_{\psi}$.

For $q = 11$ and $v \mid q$, we have the following exact sequence:
$$ 0 \rightarrow \varprojlim_n E(K_v)/(E_0(K_v),5^n) \rightarrow \Phi_v
\rightarrow \varprojlim_n E(k_v)/5^n \rightarrow 0.$$
Since $E$ has split multiplicative reduction at $11$ and the residual
index of $v$ is 3, we have
$$\varprojlim_n E(K_v)/(E_0(K_v),5^n) \simeq \F_{125} \simeq \Z/5\Z[C_3]$$
as $\Z_5[C_3]$-modules.  Also, it is easy to see that
$\varprojlim_n E(k_v)/5^n \simeq \Z/5\Z$, with trivial action of $C_3$.
Thus we get
$$\epsilon(\Phi_{11}) = \epsilon(\Z/5\Z [S_3]) \cdot
\epsilon(\mathrm{Ind}_{\Z_5[C_3]}^{\Z_5[S_3]} \Z/5\Z).$$
And thus, we get $\epsilon(\Phi_{11}) = 25(e_{\chi_0} + e_{\chi} +
e_{\psi})$.

Finally, $\epsilon(\Phi_{229}) \thicksim 1$ since $\Phi_{229}$ is trivial.
This completes the proof the lemma.
\end{proof}

Plugging the above values into (\ref{vl_inverse}), we get
$$\xi_l = (5e_{\chi_0} + e_{\chi}/5 + e_{\psi}/25) \cdot u$$
for some unit $u$ in $\Z_5[S_3]$.  But then $\beta_l(\mu_l) = u$
is a unit in $\Z_5[S_3]$ and is therefore in the image of $K_1(\Z_5[S_3])$.
This verifies the equivariant conjecture for $l = 5$.

\subsubsection{Case $l > 5$.}  Fix an $l > 5$.  In this case $\Phi_q$ is
trivial for all $q \in S_f$.  The contribution of $\tilde{\Phi_l} = \Phi_l$
is given by the following:
\begin{lemma} \label{phi_l}
Let $I$ be the inertia group corresponding to $l$.  Then one has
$$ \epsilon(\Phi_l) \thicksim \sum_{\eta \in \widehat{G}}
|l^{\dim(V_{\eta}^{I})}L_l(E \otimes \eta,1)^{-1}|_l e_{\eta},$$
where  $V_{\eta}$ is a representation of $G$ corresponding to $\eta$.
\end{lemma}

\begin{proof}
For $l = 229$ the module $\Phi_l$ is trivial.  In this case the lemma
easily follows from the table of $L$-values above.  Now, assume that
$l \neq 229$, so the inertia group $I$ is trivial.

Let $v$ be a prime in $K$ lying above $l$.  Then $\Phi_l =
\mathrm{Ind}_{\Z_l[D]}^{\Z_l[G]} E(k_v)_{l^{\infty}}$, where $D$ is the
corresponding decomposition group.  Let $f = |D|$.  The possible values
of $f$ are 1, 2 and 3.

Suppose that $f = 1$.  Since $E(\Q) = 5$ and $l > 5$, we have $E(\F_l)
\neq l$.  Thus, $\Phi_l$ is trivial and $\epsilon(\Phi_l) \thicksim 1$.
On the other hand, the Frobenius element is trivial and therefore
$L_l(E \otimes \eta,s) = (1-l^{-s})^{-\dim{\eta}}$.  This proves the
lemma when $f = 1$.

If $f = 2$ then $k_v = \F_{l^2}$.  The $l$-component of the local $L$-values
are given by
\begin{eqnarray*}
|lL_l(E,1)^{-1}|_l & = & |(1 - a_l + l)|_l = |E(\F_l)|_l, \\
|lL_l(E \otimes \chi,1)^{-1}|_l & = & |(1 + a_l + l)|_l =
|E(\F_{l^2})/E(\F_l)|_l, \\
|l^2L_l(E \otimes \psi,1)^{-1}|_l & = & |(1 + a_l + l)(1-a_l+l)|_l =
|E(\F_{l^2})|_l.\\
\end{eqnarray*}
Therefore, if $E(\F_{l^2})_{l^{\infty}}$ is trivial then the lemma holds.
If $E(\F_{l^2})_{l^{\infty}}$ is nontrivial then since $E(\F_l) \neq l$
it follows that $E(\F_{l^2})_{l^{\infty}} \simeq \Z/l\Z$, with nontrivial
action of $D$.  Since $|D| = 2$ we will suppose that $D = \{1,r\}$.
Consider the sequence
$$ 0 \longrightarrow \Z_l[S_3] \stackrel{\nu}{\longrightarrow} \Z_l[S_3]
\longrightarrow \Phi_l \longrightarrow 0,$$
where $\nu$ is multiplication by $\theta = \frac{(l+1)}{2}e -
\frac{(l-1)}{2}r$.  Since $(1-r)\theta = (1-r)$ and $(\frac{(l+1)}{2}e +
\frac{(l-1)}{2}r) \theta = l$, it follows that the above sequence is exact.
Further, we have
\begin{eqnarray*}
e_{\chi_0} \theta & = & e_{\chi_0},\\
e_{\chi} \theta & = & l e_{\chi}, \\
e_{\psi,1} \theta & = & \frac{(l+1)}{2} e_{\psi,1} + \frac{(l-1)}{2}
e_{\psi,2}, \\
e_{\psi,2} \theta & = & \frac{(l-1)}{2} e_{\psi,1} + \frac{(l+1)}{2}
e_{\psi,2}.
\end{eqnarray*}
Thus, we get that $\epsilon(\Phi_{l}) \thicksim e_{\chi_0} + le_{\chi}
+ le_{\psi}$.  This proves the lemma when $f = 2$.

Assume now that $f = 3$, so $k_v = \F_{l^3}$.  The $l$-component of the
local $L$-values are given by
\begin{eqnarray*}
|lL_l(E,1)|_l & = & |(1 - a_l + l)|_l = |E(\F_l)|_l, \\
|lL_l(E \otimes \chi,1)|_l & = & |(1 - a_l + l)|_l = |E(\F_l)|_l, \\
|l^2L_l(E \otimes \psi,1)|_l & = & |(1 + a_l + a_l^2 - l + la_l + l^2)|_l =
|E(F_{l^3})|_l/|E(\F_l)|^2_l.\\
\end{eqnarray*}
Therefore, the lemma holds if $E(F_{l^3})_{l^{\infty}}$ is trivial.  If
$E(F_{l^3})_{l^{\infty}}$ is nontrivial then since $E(\F_l) \neq l$ it
follows that $E(F_{l^3})_{l^{\infty}} \simeq \Z/l\Z$, with nontrivial
action of $D$.  Consider the sequence
$$ 0 \longrightarrow \Z_l[S_3] \stackrel{\nu}{\longrightarrow} \Z_l[S_3]
\longrightarrow \Phi_l \longrightarrow 0,$$
where $\nu$ is multiplication by $\theta = \frac{(2l+1)}{3}e -
\frac{(l-1)}{3}s - \frac{(l-1)}{3}s^2$.  Since $(1-s)\theta = (1-s)$
and $(\frac{(l+2)}{3}e + \frac{(l-1)}{3}s + \frac{(l-1)}{3}s^2) \theta = l$,
it follows that the above sequence is exact.  Further, we have
\begin{eqnarray*}
e_{\chi_0} \theta & = & e_{\chi_0},\\
e_{\chi} \theta & = & e_{\chi}, \\
e_{\psi,1} \theta & = & le_{\psi,1}, \\
e_{\psi,2} \theta & = & le_{\psi,2}.
\end{eqnarray*}
Thus, we get that $\epsilon(\Phi_{l}) \thicksim e_{\chi_0} + e_{\chi}
+ le_{\psi}$.  Thus the lemma holds when $f = 3$.  This completes the
proof of the lemma.
\end{proof}

Thus, the $l$-parts of the local $L$-factors appearing in $\xi_l$ get
cancelled with $\epsilon(\Phi_l), \epsilon(\oplus_{v \mid l} k_v)$ and
$\epsilon(R\Gamma_f(K_q,T_l(E)))$'s.  Therefore, $\xi_l = (a_{\chi_0}
e_{\chi_0} + a_{\chi} e_{\chi} + a_{\psi} e_{\psi})$ for some $a_{\chi_0},
a_{\chi}, a_{\psi} \in \Z_l^{\times}$.  This implies that $\beta_l(\mu_l)$ and it's inverse are
elements of $\Z_l[S_3]$.  Hence $\beta_l(\mu_l)$ is in the image of
$K_1(\Z_l[S_3])$.  This verifies the equivariant conjecture for $l > 5$.

The proof of Theorem \ref{intro_theorem} is now complete.
\end{proof}

\noindent {\bf Remarks.}
\begin{enumerate}
\item
Theorem \ref{intro_theorem} provides the first evidence for the
(noncommutative) equivariant conjecture for motives arising from elliptic
curves.

\item
The above argument prescribes a method to verify the equivariant
conjecture (not just numerically) in the commutative case.  However, no
methods are known to compute the $L$-values of the nonabelian twists of
an elliptic curve, and therefore the above method will only give a numerical
verification in the noncommutative case.

\item  For the abelian twists, one can write down the $L$-values in terms
of an Euler system constructed by Kato (cf. \cite{kato1}).  One can possibly
use this to get a description of the $L$-values in the fundamental line
defined via a module over an Iwasawa algebra.  This should relate the
equivariant conjecture and Kato's main conjecture as formulated in
\cite{kato1}.
\end{enumerate}

\noindent {\bf Acknowledgements.}
I would like to thank Matthias Flach for his advice and encouragement
throughout.  I also would like to thank the mathematics department at the
Indian Institute of Science for hosting me while the work was in progress.


\begin{thebibliography}{99}

\bibitem{au} Abbes, Ullmo, A propos de la conjecture de Manin pour les
courbes elliptiques modulaires, {\em Compositio Math.} 103 (1996), no. 3,
269-286.

\bibitem{artin} M. Artin, N\'{e}ron models, in: {\em Arithmetic Geometry},
edited by G. Cornell, J. H. Silverman, Springer-Verlag (1986), 213--230.

\bibitem{bk} S. Bloch, K. Kato, $L$-functions and Tamagawa numbers for
motives, {\em The Grothendieck Festschrift, Vol. I}, Progr. Math., 86
(1990) 333--400.

\bibitem{bf} D. Burns, M. Flach, Motivic $L$-functions and Galois
module structures, {\em Math. Ann.} 305 (1996), 65--102.

\bibitem{bf1} D. Burns, M. Flach, Tamagawa numbers for motives
with (noncommutative) coefficients I, {\em Doc. Math.} 6 (2001), 501--570.

\bibitem{computel} Computel, \texttt{http://maths.dur.ac.uk/$\sim$dma0td/compute
l/}

\bibitem{cremona} J. Cremona, {\em Elliptic Curves Data},
\texttt{http://www.math.nottingham.ac.uk/personal/jec/ftp/data}

\bibitem{cr} C. W. Curtis, I. Reiner, {\em Methods of Representation
Theory}, Vol. I and II, John Wiley and Sons (1987).

\bibitem{deligne} P. Deligne, Le d\'{e}terminant de la cohomology, in:
Current trends in arithmetical algebraic geometry, {\em Cont. Math.} 67
(1977), 313-346.

\bibitem{td} T. Dockchitser, Computing the special values of motivic
$L$-functions, {\em Experimental Math.} vol. 13 (2004), no. 2, 137--150.

\bibitem{flach1} M. Flach, Euler characteristics in relative $K$-groups,
{\em Bull. London Math. Soc.} 32 (2000), no. 3, 272--284.

\bibitem{kato} K. Kato, Lectures on the approach to Iwasawa theory of
Hasse-Weil $L$-functions via $B_{dR}$, Part I, In: Arithmetical algebraic
geometry, {\em Lecture Notes in Math.} 1553 (1993), 50--163.

\bibitem{kato1} K. Kato, $p$-adic Hodge theory and the values of
zeta-functions of modular forms, In: {\em Cohomologies $p$-adiques et
applications arithmetiques III}, Asterisque 295 (2004) 117-290.

\bibitem{navil2} T. Navilarekallu, On the equivariant Tamagawa number
conjecture, Caltech thesis.

\bibitem{pari} The Pari Group, PARI/GP, Version 2.1.5, 2003 Bordeaux,
available at \texttt{http://www.parigp-home.de/}.

\bibitem{serre} J. P. Serre, {\em Corps Locaux}, Hermann, Paris (1962).

\bibitem{shimura} G. Shimura, On the periods of modular forms, {\em Math.
Ann.} 229 (1977), no. 3, 211--221.

\bibitem{silverman} Silverman {\em Arithmetic on Elliptic Curves}, Graduate
Texts in Mathematics, Springer-Verlag (1992).

\bibitem{taylor} M. J. Taylor, On Frohlich's conjecture of rings of
integers of tame extensions, {\em Invent. Math.}, 63, 41-79.

\bibitem{tollis} E. Tollis, Zeros of Dedekind zeta function in the
critical strip, {\em Math. Comp.} 66 (1997), no. 219, 1295--1321.

\bibitem{zhang} S. Zhang, Heights of Heegner points on Shimura curves,
{\em Annals of Math.} 153 (2001), 27-147.

\end{thebibliography}
\end{document}